\documentclass[12pt]{article} 

\usepackage{amsmath}
\usepackage{amsthm}
\usepackage{amsfonts}
\usepackage{mathrsfs}
\usepackage{stmaryrd}
\usepackage{setspace}
\usepackage{fullpage}
\usepackage{amssymb}
\usepackage{breqn}
\usepackage{enumitem}
\usepackage{bbold} 
\usepackage{authblk}
\usepackage{comment}
\usepackage{hyperref}
\usepackage{pgf,tikz}
\usepackage{graphicx}
\usepackage{xcolor}
\usetikzlibrary{decorations.pathreplacing,arrows}
\tikzstyle{vertex}=[circle,draw=black,fill=black,inner sep=0,minimum size=3pt,text=white,font=\footnotesize]

\newtheorem*{rep@theorem}{\rep@title}
\newcommand{\newreptheorem}[2]{%
\newenvironment{rep#1}[1]{%
 \def\rep@title{#2 \ref{##1}}%
 \begin{rep@theorem}}%
 {\end{rep@theorem}}}
\makeatother

\theoremstyle{plain}
\newtheorem{theorem}{Theorem}[section]
\newreptheorem{theorem}{Theorem}

\newtheorem{proposition}[theorem]{Proposition}

\theoremstyle{definition}

\newtheorem*{proposition*}{Proposition}

\newcommand\cH{{\mathcal H}}

\newcommand{\ignore}[1]{}

\title{On the number of $A$-transversals in hypergraphs}
\author{János Barát\thanks{Research supported by ERC Advanced Grant ``GeoScape” No. 882971. and the National Research, Development and Innovation Office, grant K-131529.} \\
\small Alfr\'ed R\'enyi Institute of Mathematics, and\\
\small University of Pannonia, Department of Mathematics\\
\small \url{barat@renyi.hu}
\and
 D\'aniel Gerbner\thanks{Research supported by the National Research, Development and Innovation Office - NKFIH under the grants SNN 129364, FK 132060, and KKP-133819.}\\
\small Alfr\'ed R\'enyi Institute of Mathematics\\
\small \url{gerbner@renyi.hu}
\and
Anastasia Halfpap\\
\small University of Montana \\
\small \url{anastasia.halfpap@umontana.edu}} 

\date{}

\begin{document}

\maketitle
\begin{abstract}
A set $S$ of vertices in a hypergraph is \textit{strongly independent} if every hyperedge shares at most one vertex with $S$. We prove a sharp result for the number of maximal strongly independent sets in a $3$-uniform hypergraph analogous to the Moon-Moser theorem.

Given an $r$-uniform hypergraph $\cH$ and a non-empty set $A$ of non-negative integers, we say that a set $S$ is an \textit{$A$-transversal} of $\cH$ if for any hyperedge $H$ of $\cH$, we have \mbox{$|H\cap S| \in A$}. Independent sets are $\{0,1,\dots,r{-}1\}$-transversals, while strongly independent sets are $\{0,1\}$-transversals. Note that for some sets $A$, there may exist hypergraphs without any $A$-transversals. We study the maximum number of $A$-transversals for every $A$, but we focus on the more natural sets, $A=\{a\}$, $A=\{0,1,\dots,a\}$ or $A$ being the set of odd integers or the set of even integers.
\end{abstract}


\section{Introduction}

We use standard notation from (hyper)graph theory. 
In particular, the {\it degree} of a vertex $v$ is the number of (hyper)edges containing $v$.
Recall that a vertex set in a (hyper)graph is {\it independent} if it contains no (hyper)edge.
An independent set is {\it maximal} if it is not a proper subset of a larger independent set.
Let $MIS(G)$ denote the number of maximal independent sets in a graph $G$.
Miller and Muller \cite{mimu} and independently Moon and Moser \cite{momo} showed that for all $n$-vertex graphs $G$, $MIS(G)\le 3^{n/3}$, which is sharp as given by the vertex-disjoint union of triangles.  
More precisely, let $MIS(n)$ denote the largest $MIS(G)$ where $G$ has $n$ vertices. Then the following holds.

\begin{theorem}\label{mm}

\begin{displaymath} 
MIS(n)=\left\{ \begin{array}{l l} 3^{n/3} & \textrm{if\/ $n\equiv 0 \mod 3$},\\ 4\cdot 3^{\lfloor n/3\rfloor-1} & \textrm{if\/ $n\equiv 1 \mod 3$},\\ 2\cdot 3^{\lfloor n/3\rfloor} & \textrm{if\/ $n\equiv 2 \mod 3$}.\\\end{array}\right.\end{displaymath} 

\end{theorem}

There are several natural generalizations of this problem to hypergraphs, especially $3$-uniform hypergraphs. Cory Palmer \cite{cory} explicitly asked the following question and inspired our research.

\smallskip

{\bf Problem.} Determine the maximum number of maximal independent sets in a $3$-uniform $n$-vertex hypergraph.

\smallskip

A generalization of Theorem~\ref{mm} to hypergraphs was studied by Tomescu \cite{tom} and by Lonc and Truszczy{\'n}ski~\cite{lt}. 
They showed that in $3$-uniform $n$-vertex hypergraphs the maximum number of maximal independent sets is between roughly $1.5849^n$ and $1.6702^n$.

As it is often the case with graph theoretical notions, independent sets have multiple generalizations to hypergraphs. A set $S$ of vertices in a hypergraph is \textit{strongly independent}, if every hyperedge shares at most one vertex with $S$. We denote by $MSIS(\mathcal{H})$ the number of maximal strongly independent sets in a hypergraph $\mathcal{H}$.
We show the following analogue of Theorem~\ref{mm} for $3$-graphs.

\begin{theorem}\label{thm1}
If $\mathcal{H}$ is a $3$-uniform hypergraph on $n$ vertices, then $MSIS(\mathcal{H}) \leq g(n)$, where
 \begin{displaymath} 
g(n) = \left\{ \begin{array}{l l} 3^{n/3} & \textrm{if\/ $n\equiv 0 \mod 3$},\\ 4\cdot 3^{\frac{n-4}{3}} & \textrm{if\/ $n\equiv 1 \mod 3$},\\ 16\cdot 3^{\frac{n-8}{3}} & \textrm{if\/ $n\equiv 2 \mod 3$}.\\\end{array}\right.\end{displaymath}
 Moreover, for all $n \geq 6$, there is a $3$-uniform $n$-vertex hypergraph $\mathcal{H}$ with $MSIS(\mathcal{H}) = g(n)$.
\end{theorem}

One may consider intermediate independence notions, when a set shares a bounded number of vertices with any hyperedge. Instead, we consider the following generalization. Given an $r$-uniform hypergraph $\cH$ and a non-empty set $A$ of non-negative integers, we say that a set $S$ is an \textit{$A$-transversal} of $\cH$ if for any hyperedge $H$ of $\cH$, we have $|H\cap S| \in A$. Independent sets are $\{0,1,\dots,r{-}1\}$-transversals, while strongly independent sets are $\{0,1\}$-transversals. Note that for some sets $A$, there may exist hypergraphs without any $A$-transversals. For example, no non-bipartite graph contains a $\{1\}$-transversal. Similarly, say that an $r$-uniform hypergraph $\mathcal{H}$ is $k$-partite if $V(\mathcal{H})$ can be partitioned into $k$ strongly independent parts. Observe, no non-$r$-partite $r$-graph contains a $\{1\}$-transversal. However, we often consider $A$-transversals as generalizations of independent sets; in this case, typically it is natural to have $0\in A$, in which case $\emptyset$ is an $A$-transversal. 

We study the maximum number of $A$-transversals for every $A$, but we focus on the more natural sets, \ $A=\{a\}$, $A=\{0,1,\dots,a\}$ or $A$ being the set of odd integers or the set of even integers.

Observe that the point of studying maximal independent sets, rather than arbitrary independent sets, is that in the empty graph or empty hypergraph all the $2^n$ subsets are independent. We have the same problem for every $A$. However, it turns out, in some cases a more natural way to avoid this degenerate example is to assume that each vertex is contained in at least one (hyper)edge, i.e., there are no isolated vertices. Note that this is only possible if $n\ge r$; we omit this extra assumption later. Galvin \cite{galv} studied the maximum number of independent sets in graphs with minimum degree at least $\delta$ and showed that for sufficiently large $n$, $K_{\delta,n-\delta}$ contains the most independent sets. The threshold on $n$ was shown to be $2\delta$ in \cite{gls}.

Let us denote by $g_A^{(r)}(n)$ the maximum number of $A$-transversals in an $r$-uniform $n$-vertex hypergraph without isolated vertices. We denote by $h_A^{(r)}(n)$ the maximum number of maximal $A$-transversals in an $r$-uniform $n$-vertex hypergraph. We start by completely resolving the graph case.

\begin{theorem}\label{thm2}
All values stated in Table 1 hold.








\begin{table}[h!]\label{table}
\begin{center}
\setlength{\tabcolsep}{0.5em} 
{\renewcommand{\arraystretch}{1.7}
\begin{tabular}
{|c||c|c|c|c|c|c|c|}
 \hline
$A$  & $\{0\}$ & \{1\} & \{2\} & \{0,1\} & \{1,2\} & \{0,2\}  & \{0,1,2\} \\
 \hline
$g_A^{(2)}(n)$  & $1$    & $2^{\lfloor n/2\rfloor}$ & $1$ & $1+2^{n-1}$ & $1{+}2^{n-1}$  & $2^{\lfloor n/2\rfloor}$ & $2^n$\\
 \hline
$h_A^{(2)}(n)$      &  $1$    & $2^{\lfloor n/2\rfloor}$ & $1$  & 
$\begin{array}{c l} 3^{n/3} & \textrm{if\/ $n\equiv 0 \mod 3$},\\ 4\cdot 3^{\lfloor n/3\rfloor-1} & \textrm{if\/ $n\equiv 1 \mod 3$,}\\ 2\cdot 3^{\lfloor n/3\rfloor} & \textrm{if\/ $n\equiv 2 \mod 3$}.\end{array}$
& 1 & 1 & 1\\
\hline
\end{tabular}}
\caption{Values of $h$ and $g$ for graphs}
\end{center}
\end{table}
\end{theorem}

We remark that Theorem \ref{mm} gives the value of $h_{\{0,1\}}^{(2)}(n)$. The only other non-trivial statement of Theorem \ref{thm2} is $g_{\{0,1\}}^{(2)}(n)=g_{\{1,2\}}^{(2)}(n)=1+2^{n-1}$. The above mentioned theorem of Galvin \cite{galv} implies this statement for $n\ge 14$, and the stronger result from \cite{gls} implies the statement for $n\ge 2$.
We shall present a simple proof.

Now we list our general results. We collect some simple observations in the following proposition. 

\begin{proposition}\label{prop1}
\textbf{(i)} $g_{\{0\}}^{(r)}(n)=h_{\{0\}}^{(r)}(n)=1$.

\textbf{(ii)} If $r\in A$, then $h_{A}^{(r)}(n)=1$.

\textbf{(iii)} $g_{\{r\}}^{(r)}(n)=1$.

\textbf{(iv)} $g_{\{0,1,\dots,r\}}^{(r)}(n)=2^n$.

\textbf{(v)} $g_{\{0,r\}}^{(r)}(n)=2^{\lfloor n/r\rfloor}$.

\textbf{(vi)} If $A\subset \{0,\dots,r\}$ and $B=\{r-a: a\in A\}$, then $g_{A}^{(r)}(n)=g_{B}^{(r)}(n)$.

\textbf{(vii)} $g_{\{a\}}^{(r)}(n)=h_{\{a\}}^{(r)}(n)$ for $n\ge r$.

\end{proposition}

Finally, we list our more substantial statements for general $r$. Among other results, we determine the order of magnitude of the function $g_{A}^{(r)}(n)$ in every instance.

Let $A(i)$ denote the non-negative integers of the form $a-i$ with $a\in A$, i.e., $A(i)=\{a-i: a\in A\}\cap \mathbb{N}$. Let $f(q,i,A)=\sum_{b\in A(i)} \binom{q}{b}$. 
Observe that for a given $q$-set $Q$ there are $f(q,i,A)$ ways to pick a subset $H$ of $Q$ such that $|H|$ belong to $A(i)$. If we are given an $i$-set $X$ disjoint from $Q$, there are $f(q,i,A)$ ways to pick a subset $H$ of $Q$ such that $|H\cup X|\in A$.


\begin{theorem}\label{thm3}

\textbf{(i)} 
Let us assume $A$ and $r$ are given. We choose $q$ and $i$ with $q\le r$ and $0\le i\le r-q$ such that $f(q,i,A)^{1/q}$ is as large as possible. If the maximum is obtained at $q=p$,
then $g_{A}^{(r)}(n)=\Theta(f(p,i,A)^{n/p})$.

\textbf{(ii)} Let $A$ be the set of even integers and $B$ the set of odd integers. Then $g_{A}^{(r)}(n)=g_{B}^{(r)}(n)= 2^{\lfloor\frac{r-1}{r}n\rfloor}$.

\textbf{(iii)} Let $a<k<r$. Then $g_{\{0,1,\dots,a\}}^{(k)}(n)\ge g_{\{0,1,\dots,a\}}^{(r)}(n)\ge g_{\{0,1,\dots,a\}}^{(k)}(n-r+k)$
and  $h_{\{0,1,\dots,a\}}^{(k)}(n)\ge h_{\{0,1,\dots,a\}}^{(r)}(n)\ge h_{\{0,1,\dots,a\}}^{(k)}(n-r+k)$.
\end{theorem}

Using the $\Theta$ notation, we emphasize that $n$ goes to infinity, while $A$ and $r$ are considered fixed.

\section{Proofs}

Let us start with the proof of Theorem \ref{thm1}.
We adapt a simple inductive proof of Theorem \ref{mm} due to Chang and Jou \cite{cj} and independently Wood \cite{wood}.

\begin{proof}[Proof of {\bf Theorem \ref{thm1}}]

We can quickly verify that the theorem holds for $3$-uniform hypergraphs on $0,1,$ or $2$ vertices. Now, let $n \geq 3$, and assume that the theorem holds for all positive integers less than $n$. Let $\mathcal{H}$ be a $3$-uniform hypergraph on $n$ vertices. Let $N[v]$ denote the closed neighborhood of $v$, i.e., the set of vertices $u$ such that either $u = v$ or there is a hyperedge containing both $v$ and $u$. Let us choose $x \in V(\mathcal{H})$ with $|N[x]|$ minimal,
and define $d = |N[x]|$. Observe, any maximal strongly independent set $X$ of $\mathcal{H}$ intersects $N[x]$ at least once, since otherwise $X \cup \{x\}$ remains strongly independent and contains $X$ as a proper subset. Also observe that if $v \in N[x] \cap X$, then $X - \{v\}$ is a maximal strongly independent set in the induced subhypergraph on 
$V(\mathcal{H}) \setminus N[v]$. These observations together imply that 
$$\mathrm{MSIS}(\mathcal{H}) \leq \sum_{v \in N[x]} \mathrm{MSIS}(\mathcal{H} - N[v]),$$
since a maximal strongly independent set $X$ of $\mathcal{H}$ is counted $|X \cap N[x]|$ times by the right hand sum. 

Now, for any $v \in V(\mathcal{H})$, we have $\mathrm{MSIS}(\mathcal{H} - N[v]) \leq g(n - |N[v]|) \leq g(n - d)$, since $g$ is clearly a non-decreasing function and $d$ is the minimum size of a closed neighborhood in $\mathcal{H}$. So 

$$\mathrm{MSIS}(\mathcal{H}) \leq d \cdot g(n-d).$$

Now, we only need to show $d \cdot g(n-d) \leq g(n)$ regardless of the value of $d$. This is clearly true if $d = 1$; note also that, since $d$ represents the size of a closed neighborhood and $\mathcal{H}$ is $3$-uniform, we cannot have $d = 0$ or $d = 2$. 

When $d = 3$, note that we have $3g(n-3) = g(n)$ regardless of the value of $n$, so indeed, $3 \cdot g(n-3) \leq g(n)$.

When $d = 4$,  

 \begin{displaymath} 
4 g(n-4) = \left\{ \begin{array}{l l} \frac{64}{81} \cdot 3^{n/3} & \textrm{if\/ $n\equiv 0$ \textrm{mod} 3},\\ 4 \cdot 3^{\frac{n-4}{3}} & \textrm{if\/ $n\equiv 1$ \textrm{mod} 3},\\ 16 \cdot 3^{\frac{n-8}{3}} & \textrm{if\/ $n\equiv 2$ \textrm{mod} 3}.\\\end{array}\right.\end{displaymath}
We observe that in each case, $4g(n-4) \leq g(n)$.

When $d = 5$, 

 \begin{displaymath} 
5 g(n-5) = \left\{ \begin{array}{l l} \frac{20}{27} \cdot 3^{n/3} & \textrm{if\/ $n\equiv 0$ \textrm{mod} 3},\\ \frac{80}{27} \cdot 3^{\frac{n-4}{3}} & \textrm{if\/ $n\equiv 1$ \textrm{mod} 3},\\ 15 \cdot 3^{\frac{n-8}{3}} & \textrm{if\/ $n\equiv 2$ \textrm{mod} 3}.\\\end{array}\right.\end{displaymath}
Again, we observe in each case that $5g(n-5) \leq g(n)$.

Observe also that for any $d \geq 5$, 
$$(d-3)g(n-d+3) = 3(d-3)g(n-d) > d\cdot g(n-d).$$
Thus, if $d > 5$, we may repeatedly apply the above inequality to show that 
 \begin{displaymath} 
d\cdot g(n-d) \leq \left\{ \begin{array}{l l} 3 g(n-3) & \textrm{if\/ $d\equiv 0$ \textrm{mod} 3},\\ 4 g(n-4) & \textrm{if\/ $d\equiv 1$ \textrm{mod} 3},\\ 5 g(n-5)  & \textrm{if\/ $d\equiv 2$ \textrm{mod} 3}.\\\end{array}\right.\end{displaymath}

We conclude that for any $3$-uniform hypergraph $\mathcal{H}$ on $n$-vertices, we have $MSIS(\mathcal{H}) \leq g(n)$. It remains to describe $n$-vertex constructions achieving $g(n)$ maximal strongly independent sets.

When $n \equiv 0  \mod 3$, we take $\frac{n}{3}$ independent hyperedges. When $n \equiv 1 \mod 3$ (and $n \geq 4$), we take $\frac{n-4}{3}$ independent hyperedges and a copy of $K_4^3$. When $n \equiv 2 \mod 3$ (and $n \geq 8$), we take $\frac{n-8}{3}$ independent hyperedges and two disjoint copies of $K_4^3$.
\end{proof}

We continue by proving Proposition \ref{prop1}. Recall that it contains several simple observations.

\begin{proof}[Proof of {\bf Proposition \ref{prop1}}]
To prove \textbf{(i)}, observe that a $\{0\}$-transversal is a set of isolated vertices. There is only one such maximal set in any hypergraph, and when there are no isolated vertices, $\emptyset$ is the only $\{0\}$-transversal.

To prove \textbf{(ii)}, observe that the entire vertex set is an $A$-transversal. Thus, it is the only maximal one.

To prove \textbf{(iii)}, observe that an $\{r\}$-transversal contains all the hyperedges, and, as there are no isolated vertices in this case, it thus contains the entire vertex set. 

To prove \textbf{(iv)}, observe that every subset of vertices is a $\{0,1,\dots,r\}$-transversal.

To prove \textbf{(v)}, observe that for any connected component $C$ in a hypergraph $\cH$, a $\{0,r\}$-transversal either contains $C$ or avoids $C$. Therefore, the only choice to make is to pick some of the at most $\lfloor n/r\rfloor$ components, proving the upper bound. The lower bound is given by any hypergraph with $\lfloor n/r\rfloor$ components and no isolated vertices.
For instance, we can take $\lfloor n/r\rfloor-1$ vertex-disjoint hyperedges of size $r$, and a complete $r$-uniform hypergraph on the remaining vertices.

To prove \textbf{(vi)}, observe that the complement of an $A$-transversal is a $B$-transversal. Thus, a graph with $g_{A}^{(r)}(n)$ $A$-transversals contains at least that many $B$-transversals, showing $g_{A}^{(r)}(n)\le g_{B}^{(r)}(n)$. By symmetry, we also have $g_{B}^{(r)}(n)\le g_{A}^{(r)}(n)$, completing the proof.

To prove \textbf{(vii)}, observe first that if $S$ is an $\{a\}$-transversal and $v\not\in S$ such that $v$ is contained in a hyperedge, then $S\cup \{v\}$ is not an $\{a\}$-transversal. This means that in a hypergraph without isolated vertices, every $\{a\}$-transversal is maximal, thus by definition $g_{\{a\}}^{(r)}(n)\le h_{\{a\}}^{(r)}(n)$. 

Let $\cH$ be an $r$-uniform $n$-vertex hypergraph with $h_{\{a\}}^{(r)}(n)$ maximal $\{a\}$-transversals such that $\cH$ has the smallest number $i$ of isolated vertices among such hypergraphs. Let $I$ be the set of isolated vertices in $\cH$. 
If $i=0$, then we are done. If $i\ge r$, then we can add a hyperedge on $r$ isolated vertices to $\cH$ to obtain a new hypergraph $\cH'$. 
Now, the number of maximal $\{a\}$-transversals in $\cH'$ is at least the number of maximal $\{a\}$-transversals in $\cH$ (and will in fact be strictly larger if $0 < a < r$), but $\cH'$ has strictly fewer isolated vertices than $\cH$, a contradiction. 

If $0<i<r$, we pick a set $Q$ of $r-i$ vertices from an arbitrary hyperedge and add a new hyperedge of size $r$ containing the vertices of $Q$ and $I$. We obtain the hypergraph $\cH''$. 
Consider an $\{a\}$-transversal $S$ of $\cH$ that does not intersect $I$, and let $j:=a-|S\cap Q|\le a-r+i\le i$. Let $J$ be a $j$-element subset of $I$, then $S\cup J$ is an $\{a\}$-transversal of $\cH''$.

Observe that every maximal $\{a\}$-transversal $T$ of $\cH$ contains the vertices of $I$, and $T\setminus I$ is also an $\{a\}$-transversal of $\cH$. The set $(T\setminus I)\cup J$ is an $\{a\}$-transversal of $\cH'$, and for distinct maximal $\{a\}$-transversals of $\cH$, we obtain distinct $\{a\}$-transversals of $\cH''$ this way.
 This means that $\cH''$ contains at least as many $\{a\}$-transversals as the number of maximal $\{a\}$-transversals in $\cH$. Additionally, $\cH''$ contains no isolated vertices, completing the proof.
\end{proof}

We continue by proving Theorem \ref{thm2}, which deals with the graph case.

\begin{proof}[Proof of {\bf Theorem \ref{thm2}}]

We begin with the cases which follow quickly from established results. Observe that \textbf{(i)} of Proposition \ref{prop1} implies that $g_{\{0\}}^{(2)}(n) = h_{\{0\}}^{(2)}(n) = 1$. From \textbf{(ii)} of Proposition \ref{prop1}, we have $h_{\{2\}}^{(2)}(n) = h_{\{0,2\}}^{(2)}(n) = h_{\{1,2\}}^{(2)}(n) = h_{\{0,1,2\}}^{(2)}(n) = 1$. From \textbf{(iii)} of Proposition \ref{prop1}, we have $g_{\{2\}}^{(2)} = 1$. From \textbf{(iv)} of Proposition \ref{prop1}, we have $g_{\{0,1,2\}}^{(2)}(n) = 2^n$, and from \textbf{(v)} of Proposition \ref{prop1}, we have $g_{\{0,2\}}^{(2)}(n) = 2^{\lfloor n/2 \rfloor}$. Finally, recall that the value of $h_{\{0,1\}}^{(2)}(n)$ is given by Theorem \ref{mm}.

Now, we shall argue that $g_{\{1\}}^{(2)} = h_{\{1\}}^{(2)}(n) = 2^{\lfloor n/2 \rfloor}$. The lower bound is given by $\lfloor n/2\rfloor$ independent edges. We continue with the upper bound. Let $G$ be an $n$-vertex graph. Observe that for any $\{1\}$-transversal $S$ of $G$, the sets $S$ and $V(G)\setminus S$ form a proper 2-coloring of $G$. In particular, there are no $\{1\}$-transversals in non-bipartite graphs. There are two $\{1\}$-transversals in any bipartite connected component.
 
 If $G$ contains no isolated vertices, then there are at most $\lfloor n/2\rfloor$ connected components of $G$. Let $I$ be the set of isolated vertices in $G$. Now every maximal $\{1\}$-transversal $S$ contains $I$. Observe that $S\setminus I$ is a maximal $\{1\}$-transversal on the induced subgraph on $V(G)\setminus I$, which we denote by $G'$. Now $G'$ contains at least as many maximal $\{1\}$-transversals as $G$.
 
 As there are no isolated vertices in $V(G)\setminus I$, there are at most $\lfloor (n-|I|)/2\rfloor$ connected components in $G'$. Thus, there are at most $2^{\lfloor (n-|I|)/2\rfloor}\le 2^{\lfloor n/2\rfloor}$ maximal transversals in $G'$, completing the proof of the upper bound.

\smallskip

Finally, we argue that $g_{\{0,1\}}^{(2)}(n) = g_{\{1,2\}}^{(2)}(n) = 1 + 2^{n-1}$. Observe first that the complement of a $\{0,1\}$-transversal is a $\{1,2\}$-transversal, which proves that $g_{\{0,1\}}^{(2)}(n)=g_{\{1,2\}}^{(2)}(n)$. Note also that a $\{0,1\}$-transversal of a graph is an independent set; thus, we now simply determine the maximum number of independent sets in an $n$-vertex graph. The lower bound is given by the star $K_{1,n-1}$: the center is an independent set, and any subset of leaves is also an independent set. 

To prove the upper bound, we apply induction on $n$; the base step $n=2$ is trivial. Assume that the statement holds for $n{-1}$ and consider an $n$-vertex graph $G$ without isolated vertices. Let $v$ be a vertex of $G$, and let $u$ be a neighbor of $v$ with minimum degree among the neighbors of $v$. If $u$ has another neighbor in $G$, then we can apply the induction hypothesis on the graph $G'$ we obtain by deleting $v$ from $G$, since $G'$ does not have any isolated vertex. Thus, there are at most $1+2^{n-2}$ independent sets in $G'$. This is the number of independent sets in $G$ that do not contain $v$. The independent sets containing $v$ must avoid $u$, and thus there are at most $2^{n-2}$ such independent sets. Hence there are at most $1+2^{n-2}+2^{n-2}=1+2^{n-1}$ independent sets in $G$.

Observe that vertex $v$ was chosen arbitrarily, so we are done unless each vertex has only neighbors of degree 1, i.e., $G$ is a matching. In this case, $n$ is even and there are at most $3^{n/2}\le 1+2^{n-1}$ independent sets in $G$, completing the proof.
\end{proof}

Finally, we prove Theorem \ref{thm3}. Recall that \textbf{(i)} determines the order of magnitude of $g_A^r(n)$, \textbf{(ii)} deals with $A$ being the set of even (or odd) integers, and \textbf{(iii)} deals with $A=\{0,1,\dots,a\}$.

\begin{proof}[Proof of {\bf Theorem \ref{thm3}}]
The lower bound in \textbf{(i)} is trivial if $A=\emptyset$ or $A=\{r\}$. Otherwise, the lower bound is given by the following construction. 
We consider a set $U$ of $r-p\ge i$ vertices, and we select $\lfloor (n-2r+p)/p\rfloor$ vertex-disjoint $p$-sets of the remaining vertices. Let $\cH'$ denote the $p$-uniform hypergraph having the $p$-sets as hyperedges. Let $\cH_1$ denote the $r$-uniform hypergraph with hyperedges of the form $U\cup H$ with $H\in \cH'$.
Observe that we can extend $\cH_1$ to an $n$-vertex hypergraph by adding a set $R$ of vertices with $r\le |R|\le r+p\le 2r$.
We pick $a\in A$ with $a\neq r$ and an $a$-set $R_0$ in $R$, and the hyperedges of the $r$-uniform hypergraph $\cH_2$ are the $r$-subsets of $R$, which contain $R_0$. Let $\cH$ be the union of $\cH_1$ and $\cH_2$. Thus, $\cH$ has no isolated vertices.

Observe that $R_0$ is an $A$-transversal of $\cH_2$.
We extend $R_0$ by $A$-transversals of $\cH_1$ to obtain $A$-transversals of $\cH$.
Consider the $A$-transversals of $\cH_1$,
which contain a given $i$-subset of $U$. These $A$-transversals intersect any $p$-set $P$ of $\cH'$ in an element of $A(i)$. Let us construct such an A-transversal $S$ of $\cH_1$.
For each hyperedge of $\cH'$, there are $f(p,i,A)$ ways to select the intersection $P\cap S$. We can select the vertices from distinct hyperedges of $\cH'$ independently.
No matter how we select these intersections, if each intersection is in $A(i)$, then their union together with the chosen subset of $U$
forms an $A$-transversal of $\cH_1$. This shows that
there are at least $f(p,i,A)^{\lfloor (n-2r+p)/p\rfloor}$ 
$A$-transversals in $\cH_1$. The union of any of those $A$-transversals with $R_0$ forms an $A$-transversal in $\cH$, completing the proof of the lower bound.

Let us continue by proving the upper bound $g_{A}^{(r)}(n)\le c f(p,i,A)^{n/p}$ for some $c$ that depends on $A$ and $r$ but not on $n$. We apply induction on $n$; the base cases are covered by the option of picking $c$ large enough. Let $\cH$ be an $r$-uniform $n$-vertex hypergraph without isolated vertices. Let $v$ be a vertex of minimum degree $d\ge 1$ and let $H_1,\dots,H_d$ be the hyperedges containing $v$. Let $Q$ denote the set of vertices that are contained in $H_1\cup\cdots\cup H_d$ and no other hyperedges. Thus $v\in Q$. 

Now let us delete $Q$ and $H_1,\dots,H_d$, and let $\cH'$ be the resulting hypergraph. Observe that there are no isolated vertices in $\cH'$, as each vertex not in $Q$ is contained in another hyperedge. 
Thus we can apply the induction hypothesis to obtain that there are at most $c f(p,i,A)^{(n-|Q|)/p}$ $A$-transversals in $\cH'$.

Observe if $S$ is an $A$-transversal in $\cH$, then $S\setminus Q$ is an $A$-transversal in $\cH'$. Given $S\setminus Q$, we know the cardinality $i$ of $(S\setminus Q)\cap H_1$. Next, if $S$ is an $A$-transversal in $\cH$, then the number of vertices in $(S\cap Q)\cap H_1$ is an element of $A(i)$. This implies that for each $A$-transversal $S\setminus Q$ of $\cH'$, there are at most $f(|(S\cap Q)\cap H_1|,i,A)\le f(|Q|,i,A)$ ways to extend $S\setminus Q$ to an $A$-transversal in $\cH$. (Note that there may be fewer than that many ways if $d>1$ and some $H_j$ satisfies $|(S\setminus Q)\cap H_j|\neq i$.) Thus there are at most \[c f(|Q|,i,A)f(p,i,A)^{(n-|Q|)/p}=cf(|Q|,i,A) f(p,i,A)^{n/p}/f(p,i,A)^{|Q|/p}\] $A$-transversals in $\cH$. We have $f(|Q|,i,A)\le f(p,i,A)^{|Q|/p}$ by the definition of $p$, completing the proof.

To prove \textbf{(ii)}, let $n=ar+b$ with $b<r$, so the stated bound becomes $2^{a(r-1)+b-1}$. 
Let us show the lower bound. If $b=0$, then $n/r$ independent hyperedges give the lower bound, since there are $2^{r-1}$ ways to pick a subset of even/odd cardinality from any hyperedge. If $b>0$, we pick $(a-1)$ independent edges, and two edges $H$ and $H'$ covering together $r+b$ vertices, thus sharing $r-b$ vertices. There are $2^{r-1}$ ways to pick a subset of even/odd cardinality from $H$, and it determines whether odd or even many vertices are needed to be selected from $H'\setminus H$ so that the resulting set intersects $H'$ in an element of $A$ (or $B$). There are $2^{b-1}$ ways to pick such a set from $H'\setminus H$, completing the proof of the lower bound.

To prove the upper bound, we apply induction on $n$. 
The base case $n=r$ is trivial. Now we proceed as in the proof of the upper bound in \textbf{(i)}, defining $Q$ and $\cH'$ the same way. 
The inductive hypothesis implies there are at most $2^{\left\lfloor\frac{r-1}{r}(n-|Q|)\right\rfloor}$ $A$-transversals and $B$-transversals in $\cH'$. As in \textbf{(i)}, we know the cardinality $i$ of $(S\setminus Q)\cap H_1$. Thus, we know the number of vertices we add from $Q\cap H_1$ is a member of $A(i)$ or $B(i)$. Depending on the parity of $i$, we either need to select an odd number, or we need to select an even number of vertices from $Q\cap H_1$ to obtain an $A$-transversal (or $B$-transversal) in $\cH$.

In both cases, there are at most $2^{|Q\cap H_1|-1}\le 2^{|Q|-1}$ ways to select the correct number of vertices, thus there are at most \[2^{\lfloor\frac{r-1}{r}(n-|Q|)\rfloor+|Q|-1}=2^{(a-1)(r-1)+\lfloor\frac{r-1}{r}(r+b-|Q|)\rfloor+|Q|-1}\] $A$-transversals ($B$-transversals) in $\cH$. Observe that $\frac{r-1}{r}(r+b-|Q|)<r+b-|Q|$, thus $\lfloor\frac{r-1}{r}(r+b-|Q|)\rfloor\le r+b-|Q|-1$. Therefore, the obtained upper bound is at most \[2^{(a-1)(r-1)+r+b-|Q|-1+|Q|-1}=2^{a(r-1)+b-1},\] completing the proof.

\smallskip

To prove \textbf{(iii)}, let $\cH_1$ be an $r$-uniform $n$-vertex hypergraph and let $\cH_1'$ be its $k$-shadow, i.e., the $k$-uniform hypergraph where a $k$-set $H$ is a hyperedge if and only if $H$ is contained in a hyperedge of $\cH_1$. Let $S$ be a $\{0,1,\dots,a\}$-transversal in $\cH_1$. We claim that $S$ is also a $\{0,1,\dots,a\}$-transversal in $\cH_1'$. If it is not, i.e., if a hyperedge $H$ of $\cH_1'$ 
shares more than $a$ elements with $S$, then any hyperedge of $\cH_1$ containing $H$ also shares more than $a$ elements with $S$. This shows that $g_{\{0,1,\dots,a\}}^{(r)}(n)\le g_{\{0,1,\dots,a\}}^{(k)}(n)$.

If $S$ is a maximal $\{0,1,\dots,a\}$-transversal in $\cH_1$, then we claim that $S$ is also a maximal $\{0,1,\dots,a\}$-transversal in $\cH_1'$. If it is not, then for some $s\not\in S$, $S\cup \{s\}$ is also a $\{0,1,\dots,a\}$-transversal in $\cH_1'$. Observe that $S\cup \{s\}$ is not a $\{0,1,\dots,a\}$-transversal in $\cH_1$, thus it shares at least $a+1$ vertices with a hyperedge of $\cH_1$. Some hyperedge of $\cH_1'$ contains these $a+1$ vertices, thus shares $a+1$ vertices with $S\cup \{s\}$, a contradiction to the assumption that $S\cup \{s\}$ is a $\{0,1,\dots,a\}$-transversal in $\cH_1'$. This shows that $h_{\{0,1,\dots,a\}}^{(r)}(n)\le h_{\{0,1,\dots,a\}}^{(k)}(n)$.

Let $\cH_2$ be a $k$-uniform $(n-r+k)$-vertex hypergraph and let $\cH_2'$ be the $r$-uniform $n$-vertex hypergraph we obtain by adding the same $r-k$ new vertices to each hyperedge of $\cH_2$. If $S$ is a $\{0,1,\dots,a\}$-transversal in $\cH_2$, then clearly $S$ is also a $\{0,1,\dots,a\}$-transversal in $\cH_2'$. This shows that $g_{\{0,1,\dots,a\}}^{(r)}(n)\ge g_{\{0,1,\dots,a\}}^{(k)}(n-r+k)$.

If $S$ is a maximal $\{0,1,\dots,a\}$-transversal in $\cH_2$, then it is not necessarily maximal in $\cH_2'$, as some of the $r-k$ new elements may be added to $S$. However, let $S_1$ and $S_2$ be two distinct maximal $\{0,1,\dots,a\}$-transversals in $\cH_2$, and let $S_i'$ be a maximal $\{0,1,\dots,a\}$-transversal in $\cH_2'$ that contains $S_i$. We claim that $S_1'\neq S_2'$. Observe that there exists $s_1\in S_1\setminus S_2$ and $s_2\in S_2\setminus S_1$ by the maximality of these two transversals. Now $s_1\in S_1'\setminus S_2'$ and $s_2\in S_2'\setminus S_1'$, thus $S_1'\neq S_2'$. This shows that $h_{\{0,1,\dots,a\}}^{(r)}(n)\ge h_{\{0,1,\dots,a\}}^{(k)}(n-r+k)$.
\end{proof}

\section{Concluding remarks}

One could avoid the degenerate example of the empty (hyper)graph by different assumptions instead of having no isolated vertices. We have mentioned that in the graph case, the maximum number of independent sets in graphs with minimum degree at least $\delta>1$ was studied. Another possible assumption is that the host hypergraph is connected. The order of magnitude is determined by \textbf{(i)} of Theorem \ref{thm3} in several cases. In the remaining cases, the construction is mostly a matching. It is easy to see from our proof that similar statements hold in these cases; we only need to strengthen the assumption $q\le r$ to $q<r$ when defining $f(q,i,A)$.
Note that the largest number of \emph{maximal} independent sets in connected graphs has been determined in \cite{fur} and \cite{ggg}.

One may consider these problems for non-uniform hypergraphs as well. For the function $g$ and $1\in A$, one should assume that the hyperedges contain more than 1 vertex, as there are $2^n$ independent sets in the hypergraph consisting of $n$ hyperedges of size 1. An even better assumption is that each hyperedge contains more vertices than the largest element of $A$, as there are $2^n$ $\{0,1,\dots,a\}$-transversals among all the $k$-sets for each $k\le a$. We remark that the proof of Theorem \ref{thm3} \textbf{(iii)} extends to the non-uniform setting, and thus it is enough to determine the number of (maximal) $\{0,1,\dots,a\}$-transversals in $(a+1)$-uniform families in order to determine the order of magnitude of the analogous quantity in non-uniform families.

Finite projective planes provide prototype examples of (intersecting) $r$-uniform hypergraphs.  
In that context, a much-studied concept is the $k$-sets of $(m,n)$-type \cite{vito}.
This means $A$ consists of two non-zero elements $m$ and $n$ and the ground-set has $k$ elements.
Some values lead to beautiful constructions and some constructions lead to applications in coding theory. 
This shows that the existence problem for two-element sets $A$ has useful applications.


\end{document}